\documentclass{amsart}
\usepackage[utf8]{inputenc}

\usepackage{amsmath}
\usepackage{amssymb}
\usepackage{amsthm}
\usepackage{enumerate}
\usepackage{mathtools}
\usepackage{eq}

\theoremstyle{plain}

\newtheorem{thm}{Theorem}[section]
\newtheorem{prop}[thm]{Proposition}
\newtheorem{lem}[thm]{Lemma}

\theoremstyle{definition}

\theoremstyle{remark}

\newcommand{\F}{\mathbb{F}}

\address[F. Matsumoto]
{Department of Mathematics, Graduate School of Science,
Kyoto University, Kyoto 606-8502, Japan}
\email{matsumoto.futa.83x@st.kyoto-u.ac.jp}

\title{Relationship between Vieta-Lucas polynomials and Lucas sequences}
\author{Futa Matsumoto}

\begin{document}

\maketitle

\begin{abstract}
Let $w_n=w_n(P,Q)$ be numerical sequences which satisfy the recursion relation
\begin{equation*}
w_{n+2}=Pw_{n+1}-Qw_n.
\end{equation*}
We consider two special cases $(w_0,w_1)=(0,1)$ and $(w_0,w_1)=(2,P)$ and we denote them by $U_n$ and $V_n$ respectively.  Vieta-Lucas polynomial $V_n(X,1)$ is the polynomial of degree $n$. We show that the congruence equation $V_n(X,1)\equiv C \mod p$ has a solution if and only if $U_{(p-\epsilon)/d}(C+2,C+2)$ is divisible by $p$, where $\epsilon\in\{\pm 1\}$ depends on $C$ and $p$, and $d=\gcd(n,p-\epsilon)$.
\end{abstract}

\section{Introduction}
\label{intro}

This paper is part of the author's master thesis.
	Let $\alpha$ and $\beta$ be the two roots of the equation
    \begin{equation}
    \label{quad poly}
    x^2-Px+Q=0
    \end{equation}
    whose coefficients $P,Q$ are integers. We have
    \begin{equation}
    \label{P Q def}
    P=\alpha+\beta, Q=\alpha\beta.
    \end{equation}
We will consider two numerical sequences $U$ and $V$ defined by the equations
\begin{equation}
\label{quad def}
U_n(P,Q)=\frac{\alpha^n-\beta^n}{\alpha-\beta},\hspace{1cm} V_n(P,Q)=\alpha^n+\beta^n
\end{equation}
where $n$ is a non-negative integer.  $U_n,V_n$ are integer for any $n$, since $\alpha,\beta$ are roots of an algebraic equation and since $U_n,V_n$ are symmetric polynomials of $\alpha$ and $\beta$.
These sequences satisfy the following difference equations for $n$.
\begin{equation}
\label{def U}
U_{n+2}(P,Q)=PU_{n+1}(P,Q)-QU_n(P,Q),
\end{equation}
\begin{equation}
\label{def V}
V_{n+2}(P,Q)=PV_{n+1}(P,Q)-QV_n(P,Q).
\end{equation}
We can verify them by substituting \eqref{P Q def} and \eqref{quad def} for \eqref{def U} and \eqref{def V}.
The first few terms of $U_n$ and $V_n$ are as follows.
\begin{gather}
\label{U ex}
U_0=0,U_1=1,U_2=P,U_3=P^2-Q,\dots\\
\label{V ex}
V_0=2,V_1=P,V_2=P^2-2Q,V_3=P^3-3PQ,\dots.
\end{gather}
The sequences $U_n,V_n$ were first systematically studied by E.Lucas in \cite{Lucas}, so these sequences are called \textbf{Lucas sequences}. 

Throughout this paper, $p$ will denote an odd prime unless specified otherwise.
It was shown in \cite[pp. 344 - 345]{Carmichael2} that $U_n$ and $V_n$ are purely periodic modulo $p$ if $p\nmid Q$. From here on, we assume that $p\nmid Q$.

The \textbf{restricted period} $r(p)$ is the least positive integer $r$ such that $U_{n+r}\equiv MU_n\mod p$ for any $n\ge 0$ and some nonzero integer $M$. Since $U_0=0$ and $U_1=1$ and $U_n$ is purely periodic modulo $p$, $r(p)$ is the least positive integer $r$ such that $U_r\equiv 0 \mod p$.

The \textbf{period} $\mu(p)$ is the least positive integer $m$ such that $U_{n+r} \equiv U_n \mod p$ for any $n\ge 0$.

To determine $r(p)$, the next lemma is important. We denote Legendre symbol by $\left(\frac{*}{*}\right)$ or $(*/*)$.
\begin{lem}
{\rm(\cite[p. 424]{Lehmer} Theorem 1.9)}
\label{wp cond}

Let $p$ be a prime which does not divide $2PQ$. Let $D=P^2-4Q$. Then, $r(p)$ divides $p-\left(\frac{D}{p}\right)$.
\end{lem}

There are several studies about the behavior of $U_n$ modulo $p$. In \cite{Somer}, L. Somer considered the case where $U_n$ has a maximal limited period, that is, $r(p)=p-(D/p)$. For odd prime $p$, he proved that there exists at least one Lucas sequence $U_n(P,Q)$ such that $r(p)=p-(D/p)$. If $(D/p)=1$, then there exists exactly $\phi(p-1)\cdot (p-1)/2$  Lucas sequences $U_n(P,Q)$ where $0\le P,Q \le p-1$ and $p\nmid Q$.

In \cite{Somer2} and \cite{Somer3}, Somer and K\v{r}\'{\i}\v{z}ek studied $A_w(d)$ which is the number of times that $d$ appears in a full period of $(w)$, where $w$ is a recurrence satisfying the recursion relation
\begin{equation*}
w_{n+2}=Pw_{n+1}-Qw_n.
\end{equation*}

In this paper, we consider the number $s(p)=\frac{p-(D/p)}{r(p)}$. For divisors $d$ of $p-\left(\frac{D}{p}\right)$, whether or not $d$ divides $s(p)$ is an important question. If $d=2$, we can easily determine it as follows.
\begin{lem}
\label{2 divides sp}
{\rm(\cite[p. 313]{Somer} Lemma 3)}

$s(p)$ is divisible by $2$ if and only if $\left(\frac{Q}{p}\right)=1$.
\end{lem}
However, no definitive results are known for this question.

\vspace{1cm}

$V_m(X,1)$ is a polynomial of degree $m$. For small $m$, $V_m(X,1)$ is as follows. 
\begin{equation*}
\begin{split}
&V_1(X,1)=X,\\
&V_2(X,1)=X^2-2,\\
&V_3(X,1)=X^3-3X,\\
&V_4(X,1)=X^4-4X^2+2.
\end{split}
\end{equation*}
We call these polynomials \textbf{Vieta-Lucas polynomials}.
For prime $p$ and integer $C$, we consider the question whether or not the equation $V_m(X,1)\equiv C \mod p$ has a solution. This question is not so easy for large $m$. We found the way to answer the question by using the divisibility of $s(p)$.

\begin{thm}
\label{main thm2}
Let $n$ be a non-negative integer and $p$ be a odd prime. Let $C$ be an integer such that $p\nmid C(C^2-4)$. Let $\epsilon=\left(\frac{C^2-4}{p}\right)$ and $d=\gcd(n,p-\epsilon)$.
Then, the following $(i)$ and $(ii)$ are equivalent.
\begin{enumerate}[$(i)$]
\item There exists $x\in \mathbb{Z}$ such that $V_n(x,1)\equiv C \mod p$.
\item $p$ divides $U_{(p-\epsilon)/d}(C+2,C+2)$.
\end{enumerate}

\end{thm}

We can calculate $U_n \mod p$ rather quickly by using following formulas.
\begin{align}
U_n&=PV_n-2QV_{n-1}\\
V_n&=PV_{n-1}-QV_{n-2}\\
V_{2n}&=V_n^2-2Q^n\\
V_{2n+1}&=V_{n+1}V_{n}-2PQ^n
\end{align}
We can determine $(V_{2n+1},V_{2n})$ from $(V_{n+1},V_n)$, and $U_n$ from $(V_n,V_{n-1})$. Referring to the binary expansion of $n$, such a calculation would take $O(\log n)$ time.
Therefore, Theorem \ref{main thm2} is an useful way to determine whether or not the equation appearing in $(i)$ has a solution.

If $m=2$, since $V_m(X,1)=X^2-2$, the condition Theorem \ref{main thm2} $(i)$ means that there exists $X$ such that $X^2\equiv P^2\overline{Q}\mod p$. Therefore, it is equivalent to Lemma \ref{2 divides sp}.

To prove Theorem \ref{main thm2}, we consider the case $n\mid p-\epsilon$. In this case, we can prove Theorem \ref{main thm2} from the next theorem.

\begin{thm}
\label{main thm}
Let $p$ be an odd prime which does not divide $2PQD$. Let $m$ be a positive divisor of $p-(\frac{D}{p})$. Then, $(i)$ and $(ii)$ are equivalent.
\begin{enumerate}[$(i)$]
\item $m$ divides $s(p)$.
\item There exists $X$ in $\mathbb{Z}$ so that $V_m(X,1)\equiv P^2Q^*-2 \mod p$.
\end{enumerate}
\end{thm}

Our first goal is to prove this theorem.

\section{Lucas sequences in Finite field}
\label{finite field}

	In this section, $p$ is a fixed prime $\neq 2$. We denote the finite field of order $p$ by $\F_p$ and its algebraic closure by $\Omega_p$. By abuse of notation, integer $n$ also means its congruence class $\overline{n}$. We take $P,Q$ not in integers but in $\Omega_p\setminus\{0\}$. We denote $P^2-4Q$ by $D$ and assume $D \neq 0$.
For non-negative integer $n$, we define Lucas sequence $U_n,V_n$ in $\Omega_p$ by \eqref{quad def}. This definition makes sense because $D\neq 0$ implies $\alpha \neq \beta$.
From here on, $P,Q$ will denote elements of $\Omega_p$.

By Lemma \ref{wp cond}, if $P,Q\in \F_p$ , then
\begin{equation}
\label{U zero Fp}
U_{p-1}(P,Q)U_{p+1}(P,Q)=0.
\end{equation}
However, there exists some $P,Q\in \Omega_p\setminus \F_p$ such that \eqref{U zero Fp} holds.
The next proposition shows when \eqref{U zero Fp} holds.
\begin{prop}
\label{condition U zero}
Let $P,Q\in \Omega_p\setminus\{0\}$ and $D=P^2-4Q\neq 0$.
Then, $U_{p-1}U_{p+1}=0$ if and only if $\frac{P^2}{Q}\in \F_p$.
\end{prop}
We first prove some useful equations. 
\begin{lem}
\label{relation U V}
For any $n,m \in \mathbb{Z}$, we have
    \begin{equation}
    \label{U plus U}
        U_{n+m}+Q^mU_{n-m}=V_mU_n
    \end{equation}
    \begin{equation}
    \label{V plus V}
        V_{n+m}+Q^mV_{n-m}=V_mV_n
    \end{equation}
    \begin{equation}
    \label{U minus U}
        U_{n+m}-Q^mU_{n-m}=U_mV_n
    \end{equation}
    \begin{equation}
    \label{V minus V}
        V_{n+m}-Q^mV_{n-m}=DU_mU_n.
    \end{equation}
\end{lem}

\begin{proof}
These equations can easily be verified by substituting \eqref{P Q def} and \eqref{quad def} for \eqref{U plus U} - \eqref{V minus V}. We prove only \eqref{U plus U}. We have
\begin{equation*}
\begin{split}
U_{n+m}+Q^mU_{n-m}&=\frac{\alpha^{n+m}-\beta^{n+m}}{\alpha-\beta}
+(\alpha\beta)^m\frac{\alpha^{n-m}-\beta^{n-m}}{\alpha-\beta}\\
&=\frac{\alpha^{n+m}-\beta^{n+m}+\alpha^n\beta^m-\alpha^m\beta^n}{\alpha-\beta}\\
&=\frac{(\alpha^m+\beta^m)(\alpha^n-\beta^n)}{\alpha-\beta}\\
&=V_mU_n.
\end{split}
\end{equation*}
Remaining equations can be proved similarly.
\end{proof}

Now we can prove Proposition \ref{condition U zero}.
\begin{proof}[Proof of Proposition \ref{condition U zero}]

From Definition \ref{quad def} with $n=p$, we have
\begin{equation}
\label{Up value}
    U_p(P,Q) = (\alpha-\beta)^{p-1}=D^{(p-1)/2},
\end{equation}
\begin{equation}
\label{Vp value}
V_p(P,Q)=(\alpha+\beta)^p = P^p
\end{equation}
since $x^p+y^p=(x+y)^p$ for any $x,y$ in $\Omega_p$ and $D = (\alpha-\beta)^2$.

Also, from Lemma \ref{relation U V} with $n=(p+1)/2$ and $m = (p-1)/2$ we have

\begin{equation}
\label{VU}
    D^{(p-1)/2}+Q^{(p-1)/2}=V_{(p-1)/2}U_{(p+1)/2},
\end{equation}
\begin{equation}
\label{VV}
    P(P^{p-1}+Q^{(p-1)/2})=V_{(p-1)/2}V_{(p+1)/2},
\end{equation}
\begin{equation}
\label{UV}
    D^{(p-1)/2}-Q^{(p-1)/2}=U_{(p-1)/2}V_{(p+1)/2},
\end{equation}
\begin{equation}
\label{UU}
    P(P^{p-1}-Q^{(p-1)/2})=DU_{(p-1)/2}U_{(p+1)/2}.
\end{equation}
Moreover, from \eqref{U plus U} with $n = m$, we have
\begin{equation}
\label{U2n}
    U_{2n}=V_nU_n
\end{equation}
for any non-negative integer $n$. In particular $U_{p\pm 1}=V_{(p\pm 1)/2}U_{(p\pm 1)/2}$ with double-sign in the same order. Therefore, if $U_{p+1}U_{p-1}=0$ then the right-hand side of either \eqref{VV} or \eqref{UU} is zero, so we have
\begin{equation}
\label{PQ condition}
P^{p-1}=\pm Q^{(p-1)/2}.
\end{equation}
Thus,  if $U_{p+1}U_{p-1}=0$ we have
\begin{equation}
\left(\frac{P^2}{Q}\right)^{p-1}=1.
\end{equation}
It is equivalent to $\frac{P^2}{Q}\in\F_p$. Conversely, if $\frac{P^2}{Q}\in\F_p$, then \eqref{PQ condition} holds. It means that the left-hand side of either \eqref{VV} or \eqref{UU} is zero, so at least one of $U_{(p\pm 1)/2}$ and $V_{(p\pm 1)/2}$ is equal to 0. Therefore, we have $U_{p+1}U_{p-1}=0$. 
\end{proof}

We will need the next lemma later.
\begin{lem}
\label{V change}
Let $y\in\Omega_p$ and $y\neq 0,4$. Then,
\begin{equation}
\label{V change eq}
\frac{V_{2r}(y,y)}{y^r}=V_r(y-2,1).
\end{equation}
\end{lem}
\begin{proof}
We prove \eqref{V change eq} by induction on $r$. \eqref{V change eq} is clear for $r=0$, since $V_0(P,Q)=2$ for any $P,Q$. For $r=1$, we have
\begin{equation*}
\begin{split}
\frac{V_{2r}(y,y)}{y^r}=\frac{y^2-2y}{y}=y-2\\
V_r(y-2,1)=y-2
\end{split}
\end{equation*}
since $V_1(P,Q)=P$ and $V_2(P,Q)=P^2-2Q$.

Let $r\ge 2$ and assume that \eqref{V change eq} holds for $r-1,r-2$. By using \eqref{V plus V},
\begin{equation}
\begin{split}
\frac{V_{2r}(y,y)}{y^r}&=\frac{V_2(y,y)V_{2(r-1)}(y,y)-y^2V_{2(r-2)}(y,y)}{y^r}\\
&=\frac{(y^2-2y)V_{2(r-1)}(y,y)}{y^r}-\frac{V_{2(r-2)}(y,y)}{y^{r-2}}\\
&=\frac{(y-2)V_{2(r-1)}(y,y)}{y^{r-1}}-\frac{V_{2(r-2)}(y,y)}{y^{r-2}}\\
&=(y-2)V_{r-1}(y-2,1)-V_{r-2}(y-2,1)\\
&=V_r(y-2,1).
\end{split}
\end{equation}
We used \eqref{def V} in the last step.
Thus \eqref{V change eq} holds for any non-negative integer $r$.
\end{proof}

\section{Derived sequence and Anti-derived sequence}
\label{refinement}
In this section, we consider \textit{derived sequences} and \textit{anti-derived sequences}. For a given series $U$ if we select every $r$-th term, we obtain a new series
\begin{equation*}
0,U_r,U_{2r},U_{3r}, \cdots .
\end{equation*}
If $U_r\neq 0$, by dividing each of these by $U_r$, we obtain another series
\begin{equation*}
0,1,U_{2r}/U_r,U_{3r}/U_r,\cdots .
\end{equation*}
This series is also a Lucas sequence (see \cite[p. 437]{Lehmer} \S 4). We call it the derived sequence of order $r$ and denote it by $U^{(r)}$.

The sequence whose derived sequence of order $r$ is equal to $U$ is called the anti-derived sequence of order $r$, and is denoted by $U^{(1/r)}$. In other words, $U^{(1/r)}$ is the sequence which satisfies
\begin{equation*}
U^{(1/r)}_{nr}/U^{(1/r)}_r=U_n
\end{equation*}
for any integer $n$.
Note that the anti-derived sequence is not unique.

At first, we prove the existence of an anti-derived sequence.

\begin{lem}
Let $P,Q \in \Omega_p\setminus\{0\}$, $D=P^2-4Q\neq 0$ and $r$ a positive integer such that $p\nmid r$. Then, there exist $P_r,Q_r\in \Omega_p$ satisfy
\begin{equation}
\label{Pr Qr}
Q_r^r=Q, V_r(P_r,Q_r)=P.
\end{equation}
Moreover, $U_n(P_r,Q_r)$ is the anti-derived sequence of order $r$ for $U_n(P,Q)$.
\end{lem}

\begin{proof}
Since $\Omega_p$ is algebraically closed, there exist $Q_r\in \Omega_p$ such that $Q_r^r=Q$. Then, the equation
\begin{equation*}
V_r(X,Q_r)=P
\end{equation*}
is a polynomial equation on $X$. Therefore, there exists $P_r\in \Omega_p$ such that $V_r(P_r,Q_r)=P$.

Next, we prove $U_r(P_r,Q_r)\neq 0$.
Suppose that $U_r(P_r,Q_r)=0$. Then, $V_{2r}(P_r,Q_r)=2Q$ from \eqref{V minus V} with $n=m=r$. But also $V_{2r}(P_r,Q_r)=P^2-2Q$ from \eqref{V plus V} with $n=m=r$, so $P^2-4Q=0$. It contradicts $D\neq 0$.

We prove that the equation
\begin{equation}
\label{anti-derived exists}
\frac{U_{nr}(P_r,Q_r)}{U_r(P_r,Q_r)}=U_n(P,Q)
\end{equation}
holds for any non-negative integer $n$. \eqref{anti-derived exists} is clear for $n=0,1$. 
By \eqref{def  U} $U_n(P,Q)$ satisfies the difference equation
\begin{equation*}
U_n(P,Q)=PU_{n-1}(P,Q)-QU_{n-2}(P,Q).
\end{equation*}
On the other hand, by \eqref{U plus U} we have
\begin{equation*}
U_{nr}(P_r,Q_r)=V_r(P_r,Q_r)U_{nr-r}(P_r,Q_r)-Q_r^rU_{nr-2r}(P_r,Q_r).
\end{equation*}
Since $V_r(P_r,Q_r)=P$ and $Q_r^r=Q$, by dividing $U_r(P_r,Q_r)$, we have
\begin{equation*}
\frac{U_{nr}(P_r,Q_r)}{U_r(P_r,Q_r)}=P\frac{U_{(n-1)r}(P_r,Q_r)}{U_r(P_r,Q_r)}
-Q\frac{U_{(n-2)r}(P_r,Q_r)}{U_r(P_r,Q_r)}.
\end{equation*}
Therefore, \eqref{anti-derived exists} holds.

Since the left side of \eqref{anti-derived exists} is equal to $U_n^{(r)}(P_r,Q_r)$, $U_n(P_r,Q_r)$ is an anti-derived sequence of order $r$. 

\end{proof}

By using Proposition \ref{condition U zero} to this anti-derived sequence, we have the next lemma.

\begin{lem}
\label{main lem1}
Let $P,Q\in \Omega_p\setminus\{0\}$ and let $D\neq 0$. Let $r$ be a positive integer such that $p\nmid r$.
Then, the following (i) and (ii) are equivalent.
\begin{enumerate}[$(i)$]
\item $U^{(1/r)}_{p-1}(P,Q)U^{(1/r)}_{p+1}(P,Q)=0$ 
\item There exists $P_r,Q_r$ in $\Omega_p$ such that $P=V_r(P_r,Q_r)$, $Q=Q_r^r$ and $\frac{P_r^2}{Q_r}\in \F_p$.
\end{enumerate}

\end{lem}

Lemma \ref{main lem1} $(i)$ is equivalent to Theorem \ref{main thm} $(i)$. That is, the next lemma holds.
\begin{lem}
Let $P,Q\in \F_p\setminus\{0\}$ and let $D\neq 0$. Let $r$ be a positive integer such that $r$ divides $p-\left(\frac{D}{p}\right)$.
Then, the following $(i)$ and $(ii)$ are equivalent.
\begin{enumerate}[$(i)$]
\item $r$ divides $s(p)$.
\item $U^{(1/r)}_{p-1}(P,Q)U^{(1/r)}_{p+1}(P,Q)=0$ 
\end{enumerate}
\end{lem}
\begin{proof}
Suppose $(i)$. Then $U_{(p-(D/p))/r}=0$. Since $U_{(p-(D/p))/r}=U_{p-(D/p)}^{(1/r)}$, $(ii)$ holds.
Conversely, suppose $(ii)$. Then, $U_{p-(D/p)}^{(1/r)}=0$ or $U_{p+(D/p)}^{(1/r)}=0$ holds. Suppose that $U_{p+(D/p)}^{(1/r)}=0$. Then, $U_{r(p+(D/p))}^{(1/r)}=U_{p+(D/p)}=0$. On the other hands, $U_{p-(D/p)}(P,Q)=0$ holds from Lemma \ref{wp cond}.
Therefore, we have
\begin{equation*}
PU_p(P,Q)=-U_{p+1}(P,Q)+QU_{p-1}(P,Q)=0.
\end{equation*}
Since $P \neq 0$, it means $U_p(P,Q)=0$.
However, from \eqref{Up value} $U_p(P,Q)=D^{(p-1)/2}\neq 0$. It is a contradiction. Therefore, $U_{(p-(D/p))/r}=U_{p-(D/p)}^{(1/r)}=0$, so $(i)$ holds.
\end{proof}

\section{Proof of Theorem \ref{main thm}}

Now we can prove Theorem \ref{main thm}. It follows from the next lemma and Lemma \ref{main lem1}.
\begin{lem}
\label{main lem2}
Let $P,Q\in \Omega_l\setminus\{0\}$ and $D\neq 0$. Let $r$ be a positive integer co-prime to $p$.
Then, the following $(i)$ and $(ii)$ are equivalent.
\begin{enumerate}[$(i)$]
\item There exist $P_r,Q_r$ in $\Omega_p$ such that $P=V_r(P_r,Q_r)$, $Q=Q_r^r$ and $\frac{P_r^2}{Q_r}\in \F_p$.
\item There exists $x$ in $\F_p$ such that $V_r(x,1)=\frac{P^2}{Q}-2$.
\end{enumerate}
\end{lem}
Clearly, Lemma \ref{main lem2} $(ii)$ is equivalent to $(ii)$ of Theorem \ref{main thm} .

\begin{proof}[Proof of Lemma \ref{main lem2}]

At first, we assume ($i$).
Let $\alpha_r$ and $\beta_r$ be distinct roots of $X^2-P_rX+Q_r=0$. For any $a\neq 0$, $a\alpha_r$ and $a\beta_r$ are distinct roots of $X^2-aP_rX+a^2Q_r=0$. Therefore, we have
\begin{equation}
\label{change p q}
V_n(aP_r,a^2Q_r)=a^nV_n(P_r,Q_r)
\end{equation}
because $V_n(aP_r,a^2Q_r)=(a\alpha_r)^n+(a\beta_r)^n=a^n(\alpha_r^n+\beta_r^n)=a^nV_n(P_r,Q_r)$.
In particular, when $a=P_r/Q_r$, we have
\begin{equation}
P=V_r(P_r,Q_r)=\left(\frac{Q_r}{P_r}\right)^r V_r\left(\frac{P_r^2}{Q_r},\frac{P_r^2}{Q_r}\right).
\end{equation}

Squaring both sides,
\begin{equation}
\frac{P^2}{Q}=\frac{V_r(y,y)^2}{y^r}
\end{equation}
where $y=P_r^2/Q_r$. But from \eqref{V plus V} with $n=m=r$,
\begin{equation}
V_r(y,y)^2=V_{2r}(y,y)+2y^r.
\end{equation}
In addition, from lemma \ref{V change},
\begin{equation}
\frac{V_{2r}(y,y)}{y^r}=V_r(y-2,1).
\end{equation}
As a result, we obtain following equation
\begin{equation}
V_r(x,1)=\frac{P^2}{Q}-2
\end{equation}
where $x=y-2=P_r^2/Q_r$. So ($ii$) holds since $P_r^2/Q_r$ is in $\F_q$. 

Conversely, we assume ($ii$). 
Take $Q_r$ in $\Omega_p$ to satisfy $Q_r^r=Q$, and take $P_r$ in $\Omega_p$ to satisfy $P_r^2/Q_r-2=x$. Then,
\begin{equation*}
V_r(P_r,Q_r)^2=\left(\frac{Q_r}{P_r}\right)^{2r} V_r\left(\frac{P_r^2}{Q_r},\frac{P_r^2}{Q_r}\right)^2
=Q\left(V_r\left(\frac{P_r^2}{Q_r}-2,1\right)+2\right)=P^2.
\end{equation*}
Therefore, we have
\begin{equation}
V_r(P_r,Q_r)=\pm P
\end{equation}
If $V_r(P_r,Q_r)=P$, $(i)$ holds. Suppose that $V_r(P_r,Q_r)=-P$.
If $r$ is odd, ($i$) holds by replacing $P_r$ with $-P_r$. 
If $r=2$, $V_r(-x,1)$ is also equal to $\frac{P^2}{Q}-2$. We take $P_r'$ to satisfy $\frac{P_r'^2}{Q_r}-2=-x$. In the same way as in the proof of $(i)\Rightarrow (ii)$, we have
\begin{equation}
V_r(P_r',Q_r)=\left(\frac{Q_r}{P_r'}\right)^r V_r\left(\frac{P_r'^2}{Q_r},\frac{P_r'^2}{Q_r}\right),
\end{equation}
also. However, since $\frac{P_r'^2}{Q_r}-2=-(\frac{P_r^2}{Q_r}-2)$, $V_2(P_r',Q_r)=P_r'^2-2Q_r=-(P_r^2-2Q_r)=-V_2(P_r,Q_r)$. Thus ($i$) holds by replacing $P_r$ with $P_r'$. Note that if $r=2$, then there exist $P_r,Q_r\in \Omega_p$ such that $(i)$ holds and $\frac{P_r^2}{Q_r}-2=\pm x$.

When $r$ is even and $>2$, we proceed by induction. Let $r=2k$ and assume this lemma holds for $k$. We can verify from \eqref{V plus V} with $n=m=k$ that
\begin{equation}
V_{2k}(P,Q)=V_2(V_k(P,Q),Q^k)
\end{equation}
for any $P,Q$. In particular, we have
\begin{equation}
\frac{P^2}{Q}-2=V_k(x,1)=V_2(V_k(x,1),1).
\end{equation}
From the case $r=2$, since $V_k(x,1)\in \F_p$, there exist $P_0,Q_0$ which satisfy $P=V_2(P_0,Q_0)$, $Q=Q_0^2$, $\frac{P_0^2}{Q_0}\in \F_p$ and $V_k(x,1)=\pm (\frac{P_0^2}{Q_0}-2)$. 

In the case $V_k(x,1)=\frac{P_0^2}{Q_0}-2$, there exists $\overline{P_k},\overline{Q_k}$ which satisfy $P_0=V_k(\overline{P_k},\overline{Q_k})$, $Q_0=\overline{Q_k}^k$ and $\frac{\overline{P_k}^2}{\overline{Q_k}}\in \F_p$. \\
Since $\overline{Q_k}^{2k}= Q_0^2=Q$ and $V_{2k}(\overline{P_k},\overline{Q_k})=V_k(\overline{P_k},\overline{Q_k})^2-2\overline{Q_k}^k=P_0^2-2Q_0=P$, ($i$) holds with $P_r=\overline{P_k},Q_r=\overline{Q_k}$. 

In the case $V_k(x,1)=-(\frac{P_0^2}{Q_0}-2)$, take $D_0$ to satisfy $D_0^2=P_0^2-4Q_0$. Then $V_k(x,1) = \frac{D_0^2}{-Q_0}-2$, so there exists $\overline{P_k},\overline{Q_k}$ which satisfy $D_0=V_k(\overline{P_k},\overline{Q_k})$, $-Q_0=\overline{Q_k}^k$ and $\frac{\overline{P_k}^2}{\overline{Q_k}}\in \F_p$.\\
Since $\overline{Q_k}^{2k}=(-Q_0)^2=Q$ and $V_{2k}(\overline{P_k},\overline{Q_k})=D_0^2-(-2Q_0)=P_0^2-4Q_0=P$, ($i$) holds with these $P_r=\overline{P_k},Q_r=\overline{Q_k}$. So Lemma \ref{main lem2} proved. 
\end{proof}

\section{Proof of Theorem \ref{main thm2}}

We prove Theorem \ref{main thm2} by reducing it to Theorem \ref{main thm}.
In this section, $P,Q$ are elements of $\F_p\setminus \{0\}$. We denote $D=P^2-4Q$ and $\epsilon=(D/p)$.

\begin{lem}
\label{Vnm defo}
\begin{equation}
V_{nm}(P,Q)=V_n(V_m(P,Q),Q^m)
\end{equation}
where $n,m$ are positive integers.
\end{lem}
\begin{proof}
The cases $n=0$ and $n=1$ are trivial, so it is sufficient to verify that $V_{nm}(P,Q)$ and $V_n(V_m(P,Q),Q^m)$ satisfy the same difference equation.
From \eqref{V plus V} with replacing $n$ by $mn$, $V_{nm}(P,Q)$ satisfy the following difference equation.
\begin{equation}
V_{(n+1)m}=V_mV_{nm}-Q^mV_{(n-1)m}.
\end{equation}
It is compatible with the difference equation of $V_n(V_m(P,Q),Q^m)$.
\end{proof}

\begin{lem}
\label{V period}
The period of $V_n(P,1)$ modulo $p$ divides $p-\epsilon$. That is, 
\begin{equation*}
V_{n+(p-\epsilon)}(P,1)\equiv V_n(P,1) \mod p
\end{equation*}
for any index $n$.
\end{lem}

\begin{proof}
Since $V_0=2$ and $V_1=P$, it is sufficient to prove $V_{p-\epsilon}=2$ and $V_{p-\epsilon+1}=P$.
From \eqref{Vp value}, $V_p=P$ since $P \in \F_p$.
From \eqref{U minus U} with $m=1$, we have
\begin{equation}
V_n=U_{n+1}-QU_{n-1}=PU_n-2QU_{n-1}=2U_{n+1}-PU_n.
\end{equation}
If $\epsilon=1$, then $V_{p-1}=2U_p-PU_{p-1}=2$, since $U_p=\epsilon$ and $U_{p-\epsilon}=0$. Since $V_{p-\epsilon+1}=V_p=P$, this lemma holds for $\epsilon=1$.

Suppose that $\epsilon=-1$. Then $V_{p+1}=PU_{p+1}-2QU_p=2$ since $U_{p+1}=0, Q=1$ and $U_p=-1$. Also, 
\begin{equation*}
V_{p+2}=PV_{p+1}-QV_p=2P-P=P. 
\end{equation*}
Therefore, this lemma also holds for the case $\epsilon=-1$.
\end{proof}

\begin{lem}
\label{V bije}
Let $h$ be an integer co-prime to $p-\epsilon$. Then, the map $\F_p \to \F_p$ $(x \mapsto V_h(x,1))$ is bijective.
\end{lem}

\begin{proof}
Since $h$ is co-prime to $p-\epsilon$, there exists $k\in \mathbb{Z}$ such that $kh\equiv 1 \mod p-\epsilon$. From Lemmas \ref{Vnm defo}, \ref{V period}, we have
\begin{equation}
V_k(V_h(x,1),1)=V_{kh}(x,1)=V_1(x,1)=x.
\end{equation}
Therefore, $y\mapsto V_k(y,1)$ is inverse map of $x\mapsto V_h(x,1)$. Thus, the map $x\mapsto V_h(x,1)$ is bijective.
\end{proof}

Theorem \ref{main thm2} follows from Theorem \ref{main thm} and the next proposition.

\begin{prop}
Let $d=\gcd(n,p-\epsilon)$.
Then, the following $(i)$ and $(ii)$ are equivalent.
\begin{enumerate}[$(i)$]
\item There exists $x\in \mathbb{Z}$ such that $V_n(x,1)\equiv C \mod p$.
\item There exists $x\in \mathbb{Z}$ such that $V_d(x,1)\equiv C \mod p$.
\end{enumerate}

\end{prop}

\begin{proof}
Take $h\in \mathbb{Z}$ to satisfy $n\equiv dh \mod p-\epsilon$ and $\gcd(h,p-\epsilon)=1$. Then, $V_n(x,1)=V_{dh}(x,1)=V_d(V_h(x,1),1)$. Therefore, this proposition follows from Lemma \ref{V bije}.
\end{proof}

\bibliographystyle{amsplain}
\bibliography{Bib}

\end{document}